\documentclass[oneside,11pt,reqno]{amsart}
\usepackage{amssymb,amsmath,amsthm,bbm,enumerate,mdwlist,url,multirow,hyperref,amsthm,stmaryrd}
\usepackage[pdftex]{graphicx}
\usepackage[shortlabels]{enumitem}

\theoremstyle{definition}
\newtheorem{definition}{Definition}
\theoremstyle{theorem}
\newtheorem{proposition}[definition]{Proposition}
\newtheorem{lemma}[definition]{Lemma}
\newtheorem{theorem}[definition]{Theorem}
\newtheorem{corollary}[definition]{Corollary}

\numberwithin{equation}{section}
\numberwithin{definition}{section}
\theoremstyle{remark}
\newtheorem{remark}[definition]{Remark}

\newtheorem{example}[definition]{Example}
\newtheorem{examples}[definition]{Examples}
\def\PP{\mathsf P}
\def\MM{\mathsf M}

\def\QQ{\mathsf Q}
\def\BB{\mathcal B}
\def\AA{\mathcal A}

\def\EE{\mathcal E}
\def\GG{\mathcal G}
\def\HH{\mathcal H}

\def\FF{\mathcal F}

\def\LL{\mathcal L}

\def\Opt{\mathcal{O}}
\def\Prog{\mathrm{Prog}}
\begin{document}
\title{Independence times for iid sequences, random walks and L\'evy processes}
\author{Matija Vidmar}
\address{Department of Mathematics, University of Ljubljana, Slovenia}
\address{Institute for Mathematics, Physics and Mechanics, Ljubljana, Slovenia}
\email{matija.vidmar@fmf.uni-lj.si}
\begin{abstract}
For a sequence in discrete time having stationary independent values (respectively,  random walk) $X$, those random times $R$ of $X$ are characterized set-theoretically, for which the strict post-$R$ sequence (respectively, the process of the increments of $X$ after $R$) is independent of the history up to $R$. For a L\'evy process $X$ and a random time $R$ of $X$, reasonably useful sufficient conditions and a partial necessary condition on $R$ are given, for the process of the increments of $X$ after $R$ to be independent of the history up to $R$. 
\end{abstract}

\thanks{The author acknowledges financial support from the Slovenian Research Agency (research core funding No. P1-0222). 
}



\maketitle

 \section{Introduction}\label{section:intro}
\subsection{Motivation and overview of results}\label{subsection:motivation..}
We begin with some motivation and a synopsis of the results, introducing along the way a number of important concepts and pieces of notation.

Let $X$ be an $\mathbb{R}^d$-valued process with stationary independent increments (PSII) vanishing a.s. at zero: a random walk in discrete time or a c\`adl\`ag L\'evy process in continuous time. Now  define the process $\blacktriangle X$ and, for a random time $R$, on $\{R<\infty\}$, the processes $\Delta_RX$, $\theta_RX$, and the $\sigma$-field $\FF_R'$, by setting $(\Delta_RX)_t:=X_{R+t}-X_R$, $(\theta_RX)_t:=X_{R+t}$, and $\FF'_R$ equal to the $\sigma$-field generated by the random variables $Z_R$ with $Z$ adapted (optional) in discrete (continuous) time, finally $(\blacktriangle X)_t:=\Delta_{t}X$ (with $t$ in $\mathbb{N}_0$ or in $[0,\infty)$, as the case may be). Let us call a random time $R$ an independence (respectively, a regenerative; a Markov) time for $X$ if, on $\{R<\infty\}$, the process $\Delta_RX$ of the increments of $X$ after $R$ is independent of the history $\FF'_R$ up to $R$ (respectively, \emph{and} $\Delta_RX$ is equally distributed as $X$; given $X_R$, the post $R$-process $\theta_RX$ is distributed as $X$ started at $X_R$ and independent of $\FF_R'$). 

\begin{remark}\label{remark:distinction}
The process $\blacktriangle X$ is not yet immediately relevant to us, however it will prove useful later on. Note that $\blacktriangle X$ is a process with values in the space of $\mathbb{R}^d$-valued paths (again in discrete or continuous time, as appropriate), whereas, for a random time $R$, $\Delta_RX$ is simply a process with values in $\mathbb{R}^d$.
\end{remark}

Recall also that by definition a random time $R$ is a stopping time iff $\{R\leq t\}$ is from the past up to time $t$ for all times $t$; in discrete time it is equivalent if $\{R=t\}$ replaces $\{R\leq t\}$ in the preceding. 

It is then well-known that all stopping (but in general not all independence) times are regenerative/Markov \cite[
Theorem~3.1.2.1]{khoshnevisan} \cite[
Theorem~40.10]{sato}. It is less well-known, and somewhat remarkable, that the regenerative/Markov property already characterizes stopping times in the class of random times that are measurable with respect to $X$: under certain conditions see  \cite[Corollary~4.3]{pittenger-birth} \cite[Remark~4.14]{atkinson} in continuous  and \cite[Lemma~3.12]{jacobsen} in discrete time; we show below that this holds true in fact in full generality. 

The main contribution of this paper, however, is an investigation of independence times. Briefly, our results provide a set-theoretic characterization of (respectively, some sufficient and necessary conditions for) independence times of random walks (respectively, L\'evy processes). This then provides insight into what happens when, for a random time $R$, only the independence of $\Delta_R X$ from $\FF'_R$ is asked for, but the requirement of equality in law $\Delta_RX\sim X$ is dropped -- a natural query, given the basic importance of the property of independence in probability, and a situation that one encounters for instance with splitting at the maximum/minimum in the fundamental Wiener-Hopf factorization of PSIIs (see Examples~\ref{example:WH},~\ref{example:thin-times:W-H} and~\ref{example:at-sup} below). \label{WH:modi}

In greater detail, in the discrete time case, note that the sequence of the consecutive increments of $X$, $\Delta X:=(X_{n}-X_{n-1})_{n\in \mathbb{N}}$, has in fact stationary independent values and generates up to a completion the same filtration as $X$. In this way the investigation of independence times for $X$ is reduced to the study of the analogous times for the sequence  $\Delta X$.  This reduction, besides helping to discern what is really relevant to the argument, from what is just unnecessary `background noise', is also advantageous in that it allows the state space to become an arbitrary measurable space, enlarging the scope of the results. Thus, in Section~\ref{section:discrete}, for a sequence with stationary independent values $Y=(Y_i)_{i\in \mathbb{N}}$ taking values in an arbitrary measurable space, we characterize those random times $R$, measurable with respect to $Y$, which render the strict post-$R$ sequence $\Box_RY:=(Y_{R+i})_{i\in \mathbb{N}}$  independent of the past up to $R$  (Theorem~\ref{theorem}). Noting that a.s. for each $n\in \mathbb{N}_0$, $X_n=\sum_{i=1}^n(\Delta X)_i$, $(\Delta_R X)_n=\sum_{i=1}^n(\Box_R\Delta X)_i$ and for each $n\in \mathbb{N}$, $(\Delta X)_n=X_n-X_{n-1}$, $(\Box_R\Delta X)_{n}=(\Delta_R X)_n-(\Delta_R X)_{n-1}$, this becomes then at the same time a characterization of independence times for the random walk $X$: loosely speaking, we show that a random time $R$ \emph{of} $X$ (i.e. measurable with respect to $X$)  is an independence time for $X$, precisely when
 \begin{equation}\label{equation}
\{R=n\}=F_n\cap \{\Delta_nX\in \Gamma\}\text{ a.s. for all }n\in \mathbb{N}_0
\end{equation}
for some measurable $\Gamma$  in the path space and for some $F_n$ from the past up to time $n$, $n\in \mathbb{N}_0$. 

There is no analogue of such a reduction in continuous time and the investigation of the independence times of L\'evy processes proves much more involved. 
Notably, we provide in Section~\ref{section:cts}, sufficiency  (Proposition~\ref{proposition:necessity}\ref{proposition:necessity:ii}, dealing with thin random times, i.e. random times whose graphs are included, up to evanescence,  in the union of a denumerable family of stopping times; Proposition~\ref{proposition:sufficient:two}, dealing with strict random times, i.e. random times, whose graphs intersect the graph of any stopping time in an evanescent set only) and partial necessity (Proposition~\ref{proposition:necessity}\ref{proposition:necessity:i}) results on the independence time property. They consist in  adapting \eqref{equation} to the continuous-time setting. 
To give an informal flavor of this (I) up to technicalities, the sufficient condition for a random time $R$ to be an independence time, in the more intricate case when $R$ is strict, becomes:
\begin{equation}\label{eq:ii}
\{R\leq t\}=\int_{[0,t]}O_sdA_s \text{ for all }t\in [0,\infty)\text{ a.s.},
\end{equation}
with $O$ an optional process and with the increments of the nondecreasing right-continuous process $A$ after a given deterministic  time ``depending'' only on the increments of $X$ after that time (in words, \eqref{eq:ii} is saying that a.s. the stochastic interval $\llbracket R,\infty\rrparenthesis$ can be written as the integral of the process $O$ against the process $A$ in the usual pathwise Lebesgue-Stieltjes sense); whilst (II) thin independence times are characterized in a manner analogous to \eqref{equation}, with the deterministic times $n$ appearing in \eqref{equation} replaced by a suitable sequence of stopping times -- a condition that may also be written in the form ($\llbracket R\rrbracket$ being the graph of the thin random time $R$ of $X$):
\begin{equation}\label{eq:iii}
\llbracket R\rrbracket =O\cap \{\blacktriangle X\in \Gamma\}\text{ up to evanescence},
\end{equation}
for an optional set $O$ and a measurable set $\Gamma$ in the path space. 

\begin{remark}\label{remark:extra-explanation}
Note here that we indulge in the usual confusion between a process as a family of random variables (with values in some measurable space), defined on the underlying sample space $\Omega$, and indexed by time, on the one hand, and a process as a function defined on $\Omega\times [0,\infty)$ in continuous time, viz. on $\Omega\times \mathbb{N}_0$ in discrete time (and with values in that same measurable space), on the other. In particular, in \eqref{eq:iii}, $\blacktriangle X$ is interpreted as a map defined on $\Omega\times [0,\infty)$ (with values in path space, cf. Remark~\ref{remark:distinction}); accordingly $ \{\blacktriangle X\in \Gamma\}$ is a subset of $\Omega\times [0,\infty)$ (as are $O$ and the graph $\llbracket R\rrbracket$). In general context will always make it clear which of the two interpretations is intended.
\end{remark}

We conjecture that \eqref{eq:iii} is in fact also a characterization of independence times in the general case (i.e. when the random time $R$ of $X$ is not necessarily thin). 

Besides their theoretical appeal, conditions \eqref{equation}-\eqref{eq:ii}-\eqref{eq:iii} listed above are useful in application, for checking the independence time property: several illustrative examples of this are given in the text. 

Finally, in a minor contribution, for PSIIs,  we establish (as we have already remarked) the characterization of regenerative/Markov times as stopping times \emph{under no assumptions on the underlying space}  (Corollary~\ref{corollary:characterization-of-ST-discrete} and Theorem~\ref{theorem:characterization:STs:cts}). 

\subsection{Literature overview}\label{subsection:lit-overview}
This study falls into the 
 context of birth, death, splitting and conditional independence times of Markov processes, together with their associated path decompositions. The literature in this area is considerable: \cite[Chapter~12]{chung} \cite[Section III.9]{syski} \cite{pittenger,pittenger-sharpe,meyer,millar,jacobsen,pittenger-birth,pittenger-shih,jacobsen-solo,jacobsen-coterminal,getoor-sharpe-cooptional,sharpe-getoor} is an incomplete list. 

 In particular, and as most closely related to our investigation of independence times, random times $\tau$ of a (temporally homogeneous) Markov process $Z$ have been studied, for which  --- speaking somewhat loosely, but giving the correct flavor --- the future of $Z$ after $\tau$ is independent of the history up to $\tau$, conditionally on the $\tau$-present. Let us agree to call such $\tau$ conditional independence times for $Z$. Then in discrete time, assuming $Z$ is valued in a countable state space $J$, and defined on the canonical space of $J$-valued sequences, \cite{jacobsen} gives a set-theoretic characterization of conditional independence times \cite[Lemma~3.12]{jacobsen} (with the $\tau$-present being $Z_\tau$, the $\tau$-future being $Z_{\tau+\cdot}$ and the $\tau$-past being the $\sigma$-field generated by $W_\tau$ with $W$ adapted). In continuous time (once the notions of the `$\tau$-present' and `$\tau$-past'  
have been given a precise meaning) \cite{getoor-sharpe-cooptional,sharpe-getoor} provide several sufficient conditions for a random time to be a conditional independence time. 


\emph{But}, for a PSII $X$ and a random time $R$, the independence, on $\{R<\infty\}$, of $\Delta_RX$ from $\FF_R'$ is \emph{not} synonymous with the  independence, again on $\{R<\infty\}$, of $\theta_RX$ from $\FF'_R$ \emph{given} the $R$-present $X_R$.  Routine conditional independence calculations demonstrate that the first implies the latter (as long as $X_R$ is $\FF'_R$-measurable), but the converse fails in general  (even when $R$ is measurable with respect to $X$ -- see Example~\ref{example:not-equivalent}). Indeed, in the penultimate statement, ``the $R$-present $X_R$'' may be replaced by any $\sigma$-field contained in $\FF'_R$ and with respect to which $X_R$ is measurable (with the same relevant counter-example). Hence, unless ``the $R$-present'' is trivial, the above described  `conditional independence' results of the Markov theory do not apply directly to the investigation of whether or not $R$ is an independence time for $X$, and our results are a meaningful complement. (In special cases, indirectly, the conditional independence results may still be brought to bear even when $X_R$ is not trivial -- see Remark~\ref{remark:markov-connect}.)

Furthermore, while we consider only PSIIs/i.i.d. sequences and not more general Markov processes, the upshot is that we are able to provide the results in what is otherwise complete generality, in particular without assuming the presence of the usual Markov setup or even canonical realizations. Indeed, for this reason, we briefly include below the characterization, for PSIIs, of stopping times as regenerative/Markov times: there is no particular claim to originality here, but it is still reassuring to know, viz. the settings of \cite{pittenger-birth,atkinson,jacobsen}, that no special assumptions on the space are needed in order for this to hold true. 

Finally, we may also mention the investigation of \cite{last}: there, for a two-sided Brownian motion $(B_t)_{t\in \mathbb{R}}$, those random times $T$, measurable with respect to $B$, are characterized, for which $(B_{T+t}-B_T)_{t\in \mathbb{R}}$ is again a two-sided Brownian motion independent of $B_T$. 

\subsection{Setting} We fix, once and for all, a (not necessarily complete) probability space $(\Omega,\GG,\PP)$. The qualifiers `a.s.', `negligible', `evanescent', `with a positive probability', `independent', `law' and `completion' without reference to a probability measure, refer to the measure $\PP$. In particular, by the  completion of a  sub-$\sigma$-field $\HH$ of $\GG$, denoted $\overline{\HH}$, we shall mean the $\sigma$-field generated on $\Omega$ by $\HH$ and the negligible sets of $\GG$. The completion of a filtration is got by completing in this manner each member thereof. 

\subsection{Miscellaneous notation} \label{misc-notation} We gather here, for easy reference, some general notation used throughout the text. 

For a topological space $U$, $\mathcal{B}_U$ is its Borel $\sigma$-field (under the standard topology for the given $U$ that we will encounter). $2^U$ denotes the power set of a set $U$. The intersection of two sets $a$ and $b$ will be denoted multiplicatively: $ab:=a\cap b$. 

For measurable spaces $(A,\AA)$ and $(B,\BB)$, $\AA/\BB$ denotes the collection of $\AA$/$\BB$-measurable maps from $A$ into $B$. $\sigma_A(\ldots)$ stands for the smallest $\sigma$-field on $A$ with respect to which whichever maps that stand in between the parentheses, being all defined at least on $A$, are measurable on restriction to $A$ (the $\sigma$-field(s) on the codomain(s) being understood from context); $\sigma(\ldots):=\sigma_\Omega(\ldots)$.  For a measure $\mu$ on a $\sigma$-field $\HH$, $H\in \HH$ and  $F\in \mathcal{H}/\mathcal{B}_{[-\infty,\infty]}$, $\mu(F):=\int Fd\mu$, $\mu(F;H):=\int_H Fd\mu$, and $\mu(F\vert H)=\mu(F;H)/\mu(H)$, whenever well-defined; the parentheses in $\mu(F)$ are omitted if no ambiguity can arise. If $\HH$ is a $\sigma$-field on $X$ and $A\subset X$, $\HH\vert_A:=\{HA:H\in \HH\}$ is the trace $\sigma$-field. 

Given  a probability space $(\Theta,\HH,\QQ)$: (i) for a sub-$\sigma$-field $\mathcal{A}$ of $\HH$ and $F\in \mathcal{H}/\mathcal{B}_{[-\infty,\infty]}$ with $\QQ F^+\land \QQ F^-<\infty$, we shall denote the conditional expectation  of $F$ with respect to $\mathcal{A}$ under $\QQ$ by $\QQ(F\vert\mathcal{A})$; (ii) for a random element $X$, indiscriminately, $X_\star \QQ=\QQ_X=\QQ\circ X^{-1}$ will signify the law of $X$ under $\QQ$; (iii) for measurable spaces $(E,\mathcal{E})$ and $(J,\mathcal{J})$, random elements $X\in \mathcal{H}/\mathcal{E}$ and $Y\in \mathcal{H}/\mathcal{J}$, and a probability kernel $\mathcal{Q}=(\QQ^x)_{x\in E}$ from $(E,\mathcal{E})$ into $(J,\mathcal{J})$ (meaning: $\QQ^x$ is a probability measure on $(J,\mathcal{J})$ for each $x\in E$; the map $E\ni x\mapsto \QQ^x(G)$ belongs to $\mathcal{E}/\mathcal{B}_{[0,1]}$ for each $G\in \mathcal{J}$), we shall say that $Y$ has kernel law $\mathcal{Q}$ conditionally on $X$ under $\QQ$ if for all $G\in \mathcal{J}/\mathcal{B}_{[0,\infty]}$ and $H\in \mathcal{E}/\mathcal{B}_{[0,\infty]}$, $\QQ(G(Y) H(X))=\QQ(\QQ^X(G)H(X))$. 

Finally, for $R\in \mathcal{G}/\mathcal{B}_{[0,\infty]}$ (i.e. for a random time $R$), $\llbracket R\rrbracket$ will denote the graph of $R$; when further $A\subset \Omega$, $R_A$ will be the time equal to $R$ on $A$ and $\infty$ otherwise. More generally, in Section~\ref{section:cts} that deals with continuous time, the stochastic intervals will carry their usual meaning from the general theory of stochastic processes \cite{general-theory}, as subsets of $\Omega\times  [0,\infty)$; for instance, for a random time $R$, $\llbracket R,\infty\rrparenthesis=\{(\omega,t)\in \Omega\times [0,\infty):R(\omega)\leq t\}$, while $\llbracket R\rrbracket=\{(\omega,R(\omega)):\omega\in \Omega\}\cap (\Omega\times [0,\infty))$ (we use ``double'' parentheses, e.g. $\llbracket$ in place of $[$ etc., to emphasize that we are considering the stochastic intervals as subsets of $\Omega\times [0,\infty)$).

\section{Discrete time}\label{section:discrete}

In this section, let $(\Omega,\GG, \PP)$ be a probability space, let $\FF=(\FF_n)_{n\in \mathbb{N}_0}$ be a filtration on $\Omega$ satisfying $\FF_\infty\subset \GG$, and let $Y=(Y_n)_{n\in \mathbb{N}}$ be a sequence of random elements on  $(\Omega,\GG, \PP)$ taking values in a measurable space $(E,\EE)$, with the property that,  for all $i\in \mathbb{N}_0$, $Y_{i+1}$ is independent of $\FF_i$, is $\FF_{i+1}/\EE$-measurable and is equally distributed as $Y_1$. That is to say (with the \emph{caveat} that the filtration is defined on the index set $\mathbb{N}_0$, whilst $Y$ has the index set $\mathbb{N}$): $Y$ is an adapted process with stationary independent values relative to $\FF$. We fix an $R\in \mathcal{G}/2^{\mathbb{N}_0\cup \{\infty\}}$ (in other words, a random time $R$) with $\PP(R<\infty)>0$. Recall from Subsection~\ref{subsection:motivation..} that, on $\{R<\infty\}$, $\Box_RY=(Y_{R+r})_{r\in \mathbb{N}}$. We further set  (i) $\FF_R':=\sigma_{\{R<\infty\}}(Z_R:Z\text{ an $\FF$-adapted $(\mathbb{R},\mathcal{B}_\mathbb{R})$-valued process})$; and (ii) $\PP':=(\GG\vert_{\{R<\infty\}}\ni A\mapsto \PP(A\vert R<\infty))$. Note that when $R$ is an $\FF$-stopping time, then $\FF_R'=\FF_R\vert_{\{R<\infty\}}$ with $\FF_R=\{A\in \GG: A \{R\leq n\}\in \FF_n\text{ for all }n\in \mathbb{N}_0\}$ (and equally $\FF_\infty$ may replace $\GG$ in the latter). Also, $\FF_R'/\mathcal{B}_{[-\infty,\infty]}=\{Z_R:Z\text{ an }\mathcal{F}\text{-adapted }([-\infty,\infty],\mathcal{B}_{[-\infty,\infty]})\text{-valued process}\}$. $\LL:=\PP_Y$ will denote the law of $Y$ on the product space $(E^{\mathbb{N}},\EE^{\otimes \mathbb{N}})$; of course $\LL=\times_\mathbb{N}\PP_{Y_1}$ is the product law. 

\begin{remark}
The class of adapted processes being stable under (deterministic) stopping, we see that $\FF_R'=\sigma_{\{R<\infty\}}(Z^R:Z\text{ an $\FF$-adapted $(\mathbb{R},\mathcal{B}_\mathbb{R})$-valued process})$, i.e. $\FF'_R$ is also the $\sigma$-field generated on $\{R<\infty\}$ by the $\FF$-adapted $(\mathbb{R},\mathcal{B}_\mathbb{R})$-valued processes stopped at $R$ (with the stopped processes viewed as mapping into $(\mathbb{R}^{\mathbb{N}_0},(\mathcal{B}_\mathbb{R})^{\otimes \mathbb{N}_0})$). Indeed  $\FF_R'=\{A\in 2^{\{R<\infty\}}:A \{R=n\}=F \{R=n\}\text{ for some }F\in \FF_n\text{ for each }n\in \mathbb{N}_0\}$, and in the definition of $\FF'_R$ there is nothing special about $(\mathbb{R},\mathcal{B}_\mathbb{R})$: any measurable space $(A, \mathcal{A})$ for which there are $x\ne y$ from $A$ that are separated by $\mathcal{A}$, can replace it therein.
\end{remark}

The following technical lemma will be useful. Its proof is elementary -- we omit making it explicit.
\begin{lemma}\label{lemma}
Let  $(\Theta,\HH,\QQ)$ be a probability space. If $g\in \HH$, $\QQ(g)>0$, $\AA$ and $\BB$ are two $\QQ$-independent sub-$\sigma$-fields of $\HH$, and one can write $\QQ$-a.s. $\mathbbm{1}(g)=AB$ for $A\in \AA/\mathcal{B}_{[0,\infty]}$ and $B\in  \BB/\mathcal{B}_{[0,\infty]}$, then there exist $\QQ$-a.s. unique $a\in \AA$ and $b\in \BB$ such that $\QQ$-a.s. $g=ab$. \qed
\end{lemma}

\begin{theorem}\label{theorem}
Of the following two statements, \ref{b} implies \ref{a}. If in addition $\{R=n\}\in \overline{\FF_n\lor \sigma(\Box_nY)}$ for each $n\in \mathbb{N}_0$, then \ref{a} implies \ref{b}.

\begin{enumerate}[(a)]
\item\label{a} $\FF_R'$ is independent of $\Box_RY$ under $\PP'$.  
\item\label{b} There are a $G\in  \EE^{\otimes \mathbb{N}}$ and for all $n\in \mathbb{N}_0$ with $\PP(R=n)>0$ an $F_n\in \FF_n$, satisfying a.s. $$\{R=n\}=F_n \{\Box_nY\in G\}.$$
\end{enumerate}
When \ref{b} prevails, $G$ is $\LL$-a.s. unique, the $F_n$ are a.s. unique, and  $(\Box_RY)_\star\PP'=\PP(Y\in \cdot \vert Y\in G)=\LL(\cdot\vert G)$. 

Furthermore,  of the following two statements, \ref{d} implies \ref{c}. If in addition $\{R=n\}\in\overline{\FF_n\lor \sigma(\Box_nY)}$ for each $n\in \mathbb{N}_0$, 
then \ref{c} implies \ref{d}.

\begin{enumerate}[(i)]
\item\label{c} $\FF_R'$ is independent of $\Box_RY$ under $\PP'$ and the law of  $\Box_RY$ under $\PP'$ is $\LL$.
\item\label{d} $R$ is a stopping time relative to the completion of $\FF$. 
\end{enumerate}
\end{theorem}
The proof of Theorem~\ref{theorem} follows on p.~\pageref{proof-of-theorem} after we have given some remarks and an example. 
\begin{remark}\label{remark:discrete:RW}
\leavevmode
\begin{enumerate}
\item For sure $\{R=n\}\in \overline{\FF_n\lor \sigma(\Box_nY)}$ for each $n\in \mathbb{N}_0$, if $R\in \overline{\sigma(Y)}/\mathcal{B}_{[0,\infty]}$. 
\item Despite \ref{a} concerning only the property of independence, in order for a result of the sort of \ref{b}  to hold true, the assumption that $Y$ have, in addition to independent, \emph{stationary} values, is essential. For instance, if the process $Z=(Z_i)_{i\in \mathbb{N}}$ has independent values, with $Z_1$ and $Z_2$ taking on the values $-1$ and $1$ equiprobably, while $Y_k=1$ for $k\in \mathbb{N}_{\geq 3}$, then the random time $S:=\mathbbm{1}(Z_1=1)+2\mathbbm{1}(Z_1=-1)$ is even a stopping time of the natural filtration of $Z$, but the history of $Z$ up to $S$ is not independent of $(Z_{S+i})_{i\in \mathbb{N}}$, e.g. $\PP(S=2,Z_{S+1}=1)=\PP(S=2)\ne \PP(S=2)\PP(Z_{S+1}=1)$. 
\item\label{remark:RW} Given the discussion in Subsection~\ref{subsection:motivation..}, Theorem~\ref{theorem} has an obvious corollary for random walks. In particular, for an Euclidean space-valued random walk $X$ on $(\Omega,\GG,\PP)$, adapted and having independent increments relative to $\FF$, vanishing a.s. at zero, and for which $R\in \overline{\sigma(X)}/\mathcal{B}_{[0,\infty]}$, the condition for the independence of $\Delta_RX$ and $\FF_R'$ under $\PP'$ writes as 
\begin{equation}\label{eq:RW-independence}
\{R=n\}=F_n \{\Delta_nX\in \Gamma\}\text{ a.s. for all }n\in \mathbb{N}_0
\end{equation} for a $\Gamma\in (\mathcal{B}_{\mathbb{R}^d})^{\otimes {\mathbb{N}_0}}$ and some $F_n\in \mathcal{F}_n$, $n\in \mathbb{N}_0$. 

\item For the necessity of the conditions \ref{b} \& \ref{d}, the assumption that $\{R=n\}\in \overline{\FF_n\lor \sigma(\Box_nY)}$ for each $n\in \mathbb{N}_0$ does not in general follow from the rest of the assumptions: for instance when $\FF$ is the (completed) natural filtration of $Y$ and $R$ is non-trivial and independent of $Y$. 
\end{enumerate}
\begin{example}\label{example:WH}
It is interesting to see how the above produces the independence statement of the Wiener-Hopf factorization  for random walks. Let indeed $X$ be a real-valued random walk on $(\Omega,\GG,\PP)$, adapted to $\FF$, vanishing a.s. at zero. Assume furthermore $Z=(Z_i)_{i\in \mathbb{N}}$ is a sequence of i.i.d. random elements taking values in $\{0,1\}$, with $\PP(Z_0=1)=:q\in [0,1]$, adapted to $\FF$, independent of $X$. Note that $\Gamma:=\inf \{k\in \mathbb{N}:Z_k=1\}$ is an $\FF$-stopping time, independent of $X$, and that it has the geometric distribution on $\mathbb{N}\cup\{\infty\}$ with success parameter $q\in [0,1]$: $\PP(\Gamma=k)=q(1-q)^{k-1}$ for all $k\in \mathbb{N}$. Assume finally $Y_ {n+1}=(Y_{n+1}^1,Y_{n+1}^2)=(X_{n+1}-X_{n},Z_{n+1})$ is independent of $\FF_n$ for all $n\in \mathbb{N}_0$. 
Set next $\Gamma':=\Gamma-1$ and let $R=\sup\{k\in \{0,\ldots,\Gamma'\}\cap \mathbb{N}_0:X_k=\overline{X}_k\}$ be the last time strictly before time $\Gamma$ that $X$ is at its running supremum $\overline{X}$. We assume of course $\PP(R<\infty)>0$ (in this case it is equivalent to $\PP(R<\infty)=1$; and is automatic when $q>0$, whilst, when $q=0$, it obtains if and only if $X$ drifts to $-\infty$). Then we see that for all $n\in \mathbb{N}_0$, $\{R=n\}=\{\overline{X}_n=X_n\}\{n\leq \Gamma'\}\left\{\sum_{k=1}^l (\Box_n Y)_k^1<0\text{ for all }l\in 
\{1,\ldots,\inf\{k\in \mathbb{N}:(\Box_n Y)_k^2=1\}-1\}\cap 
\mathbb{N} \right\}$. Hence Theorem~\ref{theorem} applies and tells us that, under $\PP'$, $\FF_R'$ (in particular the stopped process $X^R$) is independent of $\Box_RY$ (in particular of $\Delta_RX$). Thus, when $q>0$ and hence $R\leq \Gamma'<\infty$ a.s., the pair $(R,\overline{X}_{\Gamma' })$ is seen to be independent of the pair $(\Gamma'-R,\overline{X}_{\Gamma'}-X_{\Gamma'})$, which is the independence statement of the Wiener-Hopf factorization for random walks. Note that Theorem~\ref{theorem} also gives the law of the post-$R$ increments of $X$: on $\{R<\infty\}$, $\Delta_R X$ behaves as $X$ conditioned not to \emph{return} to zero up to strictly before an independent geometric random time on $\mathbb{N}\cup \{\infty\}$ with success parameter $q$. (Of course, in the above, we could take, mutatis mutandis, $R=\inf\{k\in \{0,\ldots,\Gamma'\}\cap \mathbb{N}_0:X_k=\overline{X}_{\Gamma'}\}$ and essentially the same results would follow. We leave it to the interested reader to make the eventual differences explicit.)
\end{example}
\end{remark}
\begin{proof}[Proof of Theorem~\ref{theorem}]\label{proof-of-theorem}
By a monotone class argument, $\Box_nY$ is independent of $\FF_n$ for every $n\in \mathbb{N}_0$ (and has clearly the same law as $Y$). Let $M:=\{n\in \mathbb{N}_0:\PP(R=n)>0\}$. 

Assume \ref{b}. We see that for any $H\in  \EE^{\otimes \mathbb{N}}/\mathcal{B}_{[0,\infty]}$, any $\FF$-adapted process $Z$ with values in $([0,\infty],\mathcal{B}_{[0,\infty]})$, and then each $n\in M$, 
 $$\PP(Z_RH(\Box_RY);R=n)=\PP(Z_nH(\Box_nY);R=n)=\PP(Z_nH(\Box_nY);F_n\{\Box_nY\in G\})$$
$$=\PP(Z_n;F_n)\PP(H(\Box_nY);\Box_nY\in G)=\PP(Z_n ;F_n)\PP(H(Y); Y\in G).$$
 Summing over $n\in M$ implies $\PP(Z_RH(\Box_RY);R<\infty)=\PP\left( \sum_{n\in M}Z_n \mathbbm{1}_{F_n}\right)\PP(H(Y);Y\in G)$. From this we obtain, taking $Z\equiv 1$:   $\PP(H(\Box_RY);R<\infty)=\PP(\sum_{n\in M}\mathbbm{1}_{F_n})\PP(H(Y);Y\in G)$; taking $H\equiv 1$:  $\PP(Z_R;R<\infty)=\PP(\sum_{n\in M} Z_n\mathbbm{1}_{F_n})\PP (Y\in G)$; taking $Z\equiv 1\equiv H$: $\PP(R<\infty)=\PP(\sum_{n\in M}\mathbbm{1}_{F_n})\PP (Y\in G)$. Then  $\PP(H(\Box_RY);R<\infty)\PP(Z_R;R<\infty)=
\PP(Z_RH(\Box_RY);R<\infty)\PP(R<\infty)$,  whence the desired independence in \ref{a} follows. We also obtain $\PP(H(\Box_RY)\vert R< \infty)=\PP(H(Y)\vert Y\in G)$ for all $H\in\EE^{\otimes \mathbb{N}}/\mathcal{B}_{[0,\infty]}$.  If moreover $R$ is a stopping time relative to the completion of $\FF$, i.e. $G=E^{\mathbb{N}}$ $\LL$-a.s., then this entails $(\Box_RY)_\star\PP'=\LL$.

Now assume \ref{a} and that $\{R=n\}\in \overline{\FF_n\lor \sigma(\Box_nY)}$ for each $n\in \mathbb{N}_0$. 
We first  show that 
\begin{enumerate}[(I)]
\item\label{i} For each $n\in M$, there exist $A_n\in \FF_n/\mathcal{B}_{[0,\infty]}$ and $B_n\in \sigma(\Box_n(Y))/\mathcal{B}_{[0,\infty]}$, such that a.s. $\mathbbm{1}(R=n)=A_nB_n$.  
\suspend{enumerate}

Suppose \emph{per absurdum} that  \ref{i} fails for some $n\in M$. Then it must be the case that with a positive probability, the equality $\PP(R=n)\mathbbm{1}(R=n)=\PP(R=n\vert \FF_n)\PP(R=n\vert \Box_nY)$ fails. By the independence of $\FF_n$ and $\Box_nY$, the assumption $\{R=n\}\in \overline{\FF_n\lor \sigma(\Box_nY)}$, and a monotone class argument, this implies that there are $A\in \FF_n$ and $B\in \sigma(\Box_nY)$ such that $$\PP(R=n)\PP(R=n,A,B) \ne \PP(R=n,A)\PP(R=n,B).$$ The assumption of \ref{a} implies that $\PP(R<\infty)\PP(\tilde{A},\tilde{B},R=n)=\PP(\tilde{A},R=n)\PP(\tilde{B},R<\infty)$ for any $\tilde{A}\in \FF_n$ and $\tilde{B}\in \sigma_{\{R<\infty\}}(\Box_RY)$ (for, $\mathbbm{1}_{\llbracket n\rrbracket}\mathbbm{1}_{\tilde{A}}$ is $\FF$-adapted, so $\tilde{A}\{R=n\}\in \FF_R'$). Taking $\tilde{A}=\Omega$ allows to conclude that
\begin{equation}\label{eq:useful}
\PP(R<\infty)\PP(\tilde{B},R=n)=\PP(\tilde{B},R<\infty)\PP(R=n),
\end{equation}
hence $$\PP(R=n)\PP(\tilde{A},\tilde{B},R=n)=\PP(\tilde{A},R=n)\PP(\tilde{B},R=n).$$ Now, $\sigma(\Box_nY)\vert_{\{R=n\}}=\sigma_{\{R<\infty\}}(\Box_RY)\vert_{\{R=n\}}$. Therefore $B\{R=n\}=\tilde{B}\{R=n\}$ for some $\tilde{B}\in \sigma_{\{R<\infty\}}(\Box_RY)$. Taking finally $\tilde{A}=A$ yields a contradiction.

By Lemma~\ref{lemma} and \ref{i}, for each $n\in M$, there are an a.s. uniquely determined $F_n\in \FF_n$ and an $\LL$-a.s. uniquely determined $G_n\in \EE^{\otimes \mathbb{N}}$, such that a.s. $\{R=n\}=F_n \{\Box_nY\in G_n\}$. We now show that 

\resume{enumerate}[]
\item[(II)]  $G_n=G_m$ $\LL$-a.s. whenever $\{n,m\}\subset M$. 
\end{enumerate}

Let $\{m,n\}\subset M$. We compute, for any $H\in \EE^{\otimes \mathbb{N}}/\mathcal{B}_{[0,\infty]}$ (using \eqref{eq:useful} in the last equality):
\footnotesize
$$\PP(R=m)\PP(H(Y);Y\in G_m)=\PP(R=m)\PP(H(\Box_mY);\Box_mY\in G_m)=\PP (F_m)\PP(\Box_mY\in G_m)\PP(H(\Box_mY);\Box_mY\in G_m)$$ $$=\PP (H(\Box_mY);F_m \{\Box_mY\in G_m\})\PP(Y\in G_m)=\PP(H(\Box_mY);R=m)\PP(Y\in G_m)$$ $$=\PP(H(\Box_R Y);R=m)\PP(Y\in G_m)=\PP(Y\in G_m)\frac{\PP(R=m)\PP(H(\Box_R Y);R<\infty)}{\PP(R<\infty)},$$ \normalsize i.e. $\PP(R<\infty)\PP(H(Y);Y\in G_m)=\PP(Y\in G_m)\PP(H(\Box_R Y);R<\infty)$. Then $\PP(H(Y)\vert Y\in G_m)=\PP(H(\Box_RY)\vert R<\infty)=\PP(H(Y)\vert Y\in G_n)$ for all  $H\in \EE^{\otimes \mathbb{N}}/\mathcal{B}_{[0,\infty]}$, which implies that $\LL(G_n\triangle G_m)=1$ [take $H=\mathbbm{1}_{G_n}$ and $H=\mathbbm{1}_{G_m}$]. 

If we are moreover given that $(\Box_RY)_\star\PP'=\LL$, then with $G$ as in \ref{b}, by what we have already proven,  $\LL(G)=\PP(\Box_RY\in G\vert R<\infty)=\PP(
Y\in G\vert Y \in G)=1$. 
\end{proof}
Recall the notation $\Delta_RX$ and $\theta_RX$ from Subsection~\ref{subsection:motivation..}.
\begin{corollary}\label{corollary:characterization-of-ST-discrete}
Let $d\in \mathbb{N}$ and let $X$ be an $(\mathbb{R}^d,\mathcal{B}_{\mathbb{R}^d})$-valued random walk on $(\Omega,\GG,\PP)$, adapted and having independent increments relative to the filtration $\FF$, vanishing a.s. at zero. For $x\in \mathbb{R}^d$, let $\PP^x$ be the law of $x+X$ under $\PP$; $\mathcal{P}:=(\PP^x)_{x\in \mathbb{R}^d}$. Of the following statements, \ref{RWd} implies \ref{RWc}, \ref{RWc} and \ref{RWc'} are equivalent. If in addition $\{R=n\}\in\overline{\FF_n\lor \sigma(\Delta_nX)}$ for each $n\in \mathbb{N}_0$, 
then \ref{RWc} implies \ref{RWd}.

\begin{enumerate}[(i)]
\item\label{RWc'} Under $\PP'$, conditionally on $X_R$, $\theta_R X$ is independent of $\FF_R'$ and has kernel law $\mathcal{P}$. 
\item\label{RWc} $\FF_R'$ is independent of $\Delta_RX$ under $\PP'$ and the law of  $\Delta_RX$ under $\PP'$ is $\PP^0$.
\item\label{RWd} $R$ is a stopping time relative to the completion of $\FF$. 
\end{enumerate}
\end{corollary}
\begin{proof}
The equivalence of \ref{RWc'}  and \ref{RWc} is by standard manipulation of conditional independence, using the fact that $X_R\in \FF'_R/\mathcal{B}_{\mathbb{R}^d}$.  The rest follows by Theorem~\ref{theorem} applied to the sequence $\Delta X$ of the consecutive increments of $X$ (viz. the reduction of $X$ to $\Delta X$ of Subsection~\ref{subsection:motivation..}). 
\end{proof}

\begin{remark}\label{remark:cf}
Assume the random walk $X$ of Remark~\ref{remark:discrete:RW}\eqref{remark:RW} takes values in a denumerable set $J$ and is the coordinate process on the canonical filtered space $(J^{\mathbb{N}_0},\sigma(X),\FF^X)$ of $J$-valued sequences (with $\FF^X$ the natural filtration of $X$). Compare \eqref{eq:RW-independence} with the condition of  \cite[Lemma~3.12]{jacobsen} for $\FF_R'$ to be independent of $\theta_RX$ conditionally on $X_R$ under $\PP'$ (i.e. with the condition for $R$ to be a  `conditional independence time' in the terminology of \cite{jacobsen}), namely that there should be a $G\in  \EE^{\otimes \mathbb{N}_0}$ and $F_n\in \FF_n$ for $n\in \mathbb{N}_0$, satisfying 
\begin{equation}\label{eq:jacobsen}
\{R=n\}=F_n \{\theta_nX \in G\}\text{ a.s.  for all }n\in \mathbb{N}_0.
\end{equation} \eqref{eq:RW-independence} implies the latter, but the converse fails in general, as the next example demonstrates. Note also that the proof method of  \cite{jacobsen} for establishing condition \eqref{eq:jacobsen} ---  working on atoms by exploiting the canonical setting --- is quite different from our method for establishing condition \eqref{eq:RW-independence}. 
\end{remark}

\begin{example}\label{example:not-equivalent}
Retain the provisions of Remark~\ref{remark:cf}, with further $J=\mathbb{Z}$, 
$X$ drifting to $-\infty$, and $\PP(X_1<-1)$, $\PP(X_1=-1)$, $\PP(X_1>0)$ all positive. Let $R$ be equal to (i) the last time $X$ is at its running supremum $\overline{X}$ and $X$ is not equal to $0$, on the event that there is such a time; (ii) the last time $X$ is at  $\overline{X}$ and $X$ is equal to $0$ and $X$ jumps down by $1$ on its next step, on the event that there is such a time; (iii) $\infty$ otherwise. Then for $n\in \mathbb{N}_0$, \footnotesize $$\{R=n\}=\{\overline{X}_n=X_n\} \left[\{X_n=0,X_{n+1}=-1,X_{n+m}<0\text{ for }m\in \mathbb{N}_{\geq 2}\}\cup \{X_n\ne 0,X_{n+m}<X_n\text{ for }m\in \mathbb{N}\}\right].$$\normalsize Notice that  $\PP(R<\infty)>0$ thanks to the assumptions on $X$. It is intuitively clear that $R$ is a conditional independence time, but not an independence time for $X$. Formally, it follows from \eqref{eq:jacobsen} that $\Delta_R X$ is independent of $\FF_R'$ given $X_R$ under $\PP'$, and from \eqref{eq:RW-independence}, that this fails if the proviso `given $X_R$' is dropped. Indeed, from the properties of conditional independence it is seen easily, that not only in this example is $\theta_R X$ independent of $\FF'_R$ given $X_R$ under $\PP'$, but that in the latter statement any $\sigma$-field containing $\sigma_{\{R<\infty\}}(X_R)$ and contained in $\FF'_R$ may replace $X_R$. 
\end{example}

\section{Continuous time}\label{section:cts}

In this section, let $d\in \mathbb{N}$ and let $X=(X_t)_{t\in [0,\infty)}$ be a L\'evy process  on  $(\Omega,\GG, \PP)$, relative to a filtration $\FF=(\FF_t)_{t\in [0,\infty)}$ on $\Omega$ satisfying $\FF_\infty\subset \GG$, taking values in $(\mathbb{R}^d,\mathcal{B}_{\mathbb{R}^d})$: so $X$ is adapted and has stationary independent increments relative to $\FF$, vanishes at zero a.s., and has c\`adl\`ag paths. Denote by $\Opt$ (respectively, $\Prog$) the optional (respectively, progressive) $\sigma$-field, i.e. the $\sigma$-field on $\Omega\times [0,\infty)$ generated by the $\FF$-adapted c\`adl\`ag (respectively, $\FF$-progressively measurable) $(\mathbb{R},\mathcal{B}_\mathbb{R})$-valued processes. \label{cts:firstpara} Then $O\in \mathcal{O}/\mathcal{B}_\mathbb{R}$ is just another way of saying that $O$ is a real-valued optional process; to say that $O\in \mathcal{O}$ means that $O$ is an optional subset of $\Omega\times [0,\infty)$. We fix next an $R\in \GG/\mathcal{B}_{[0,\infty]}$ with $\PP(R<\infty)>0$. Recall from Subsection~\ref{subsection:motivation..} that $\llbracket R\rrbracket=\{(\omega,t)\in \Omega\times [0,\infty):t=R(\omega)\}$ is the graph of $R$ while, on $\{R<\infty\}$, $\Delta_RX=(X_{R+t}-X_R)_{t\in [0,\infty)}$ (this notation is to be retained for processes and random times other than the given $X$ and $R$) and $\theta_RX=(X_{R+t})_{t\in [0,\infty)}$. We further set  (i) $\FF_R':=\sigma_{\{R<\infty\}}(Z_R:Z\in \Opt/\mathcal{B}_{\mathbb{R}})$ (a notation to be retained for random times other than the given $R$) and $\FF_{R+}':=\sigma_{\{R<\infty\}}(Z_R:Z\in \Prog/\mathcal{B}_{\mathbb{R}})$; and (ii) $\PP':=(\GG\vert_{\{R<\infty\}}\ni A\mapsto \PP(A\vert R<\infty))$. When $R$ is an $\FF$-stopping time, then $\FF_{R+}'=\FF_R'=\FF_R\vert_{\{R<\infty\}}$ with $\FF_R=\{A\in \GG: A\{R\leq t\}\in \FF_t\text{ for all }t\in [0,\infty)\}$ (and equally $\FF_\infty$ may replace $\GG$ in the latter). Besides, $\FF_{R+}'/\mathcal{B}_{[-\infty,\infty]}=\{Z_R:Z\in\Prog/\mathcal{B}_{[-\infty,\infty]}\}$ and $\FF_{R}'/\mathcal{B}_{[-\infty,\infty]}=\{Z_R:Z\in \mathcal{O}/\mathcal{B}_{[-\infty,\infty]}\}$. 
\begin{remark}\label{remark:various}
When $R$ is not a stopping time, various (other) interpretations of the notion of the $\sigma$-field $\FF_R$ of the past up to $R$ appear in the literature, for instance:
\begin{enumerate}
\item $\FF_R$ might be taken to consist of those $A\in\FF_\infty$, for which, given any $t\in [0,\infty)$, there is an $A_t\in \FF_t$ with $A\{R<t\}=A_t\{R<t\}$. When $R$ is an honest time $\FF_R\vert_{\{R<\infty\}}=\FF_{R+}'$ \cite[Proposition~3.1]{meyer}.
\item One can let $\FF_R$ be generated by $\FF_S\vert_{\{S\leq R\}}$ as $S$ ranges over the $\FF$-stopping times. According to \cite{pittenger-birth} (that quotes \cite{indiana}) $\FF_R\vert_{\{R<\infty\}}=\FF_R'$.
\item One can take for $\FF_R$, at least when $\FF$ is the (completed) natural filtration of $X$, the (completed) initial structure $\sigma(X^R,R)$.  See \cite[Section~4]{atkinson} for some connections between the latter and the natural filtration of $X$. 
\end{enumerate} 
Finally note that a  natural related problem to the study of what we have called independence times would be to investigate conditions under which, on $\{R<\infty\}$, $(X_{R+t}-X_{R-})_{t\in [0,\infty)}$ ($X_{0-}:=X_0$) is independent of $\sigma_{\{R<\infty\}}(Z_R:Z\in \mathcal{P}/\mathcal{B}_{\mathbb{R}})$, $\mathcal{P}$ being the predictable $\sigma$-field (cf. for instance the statement of the independence property of the Wiener-Hopf factorization for L\'evy processes in the case that $0$ is regular for itself for the drawdown process \cite[Lemma~VI.V(ii)]{bertoin}). We do not pursue this problem here. \label{added:predictable}
\end{remark}
We set $\LL:=X_\star\PP$, the law of $X$ on the space $(\mathbb{D},\mathcal{D})$, where $\mathbb{D}$ are c\`adl\`ag paths mapping $[0,\infty)\to \mathbb{R}^d$ and $\mathcal{D}$ is the $\sigma$-field generated by the canonical projections.

The following lemma describes how the property of being an independence time `separates' over optional sets and also over sets that `depend only on the incremental future'. 

\begin{lemma}\label{lemma:restriction}\leavevmode
\begin{enumerate}[(I)]
\item\label{restriction:I} Suppose $\FF_R'$ is independent of $\Delta_RX$ under $\PP'$, and that $S\in \mathcal{G}/\mathcal{B}_{[0,\infty)}$ satisfies $\PP(S<\infty)>0$ and $\llbracket S\rrbracket=A\llbracket R\rrbracket$ with $A\in \Opt\cup\sigma_{\Omega\times [0,\infty)}(\blacktriangle X)$.  Set $\LL':=(\Delta_RX)_\star\PP'$ and $\PP^{:}:=(\mathcal{G}\vert_{\{S<\infty\}}\ni F\mapsto \PP(F\vert S<\infty))$.  
\begin{enumerate}[(a)]
\item\label{restriction:a} If $A\in \Opt$, then $\Delta_SX$ is independent of $\FF_S'$ under $\PP^:$ and $(\Delta_SX)_\star \PP^:=\LL'$. 
\item\label{restriction:b} If $A=(\blacktriangle X)^{-1}(\Gamma)$ for a $\Gamma\in \mathcal{D}$, then again $\Delta_SX$ is independent of $\FF_S'$ under $\PP^:$ and $(\Delta_SX)_\star \PP^:=\LL'(\cdot\vert \Gamma)$. 
\end{enumerate}
\item\label{restriction:II} Conversely, let $\{R_1,R_2\}\subset \mathcal{G}/\mathcal{B}_{[0,\infty]}$, $\PP(R_i<\infty)>0$ for $i\in \{1,2\}$. Set $\PP^i:=( \mathcal{G}\vert_{\{R_i<\infty\}}\ni F\mapsto \PP(F\vert R_i<\infty))$ and suppose $\Delta_{R_i}X$ is independent of $\FF_{R_i}'$ under $\PP^i$ for $i\in \{1,2\}$. 
\begin{enumerate}[(i)]
\item\label{restriction:i} 
If there is an $A\in \Opt$ with, up to evanescence, $\llbracket R_1 \rrbracket=\llbracket R \rrbracket A$, $\llbracket R_2 \rrbracket=\llbracket R \rrbracket A^c$, and if  $\MM:=(\Delta_{R_1}X)_\star \PP^1=(\Delta_{R_2}X)_\star \PP^2$, then $\FF_R'$ is independent of $\Delta_RX$ under $\PP'$ and $(\Delta_RX)_\star \PP'=\MM$.
\item\label{restriction:ii}  If there is a $\Gamma\in \mathcal{D}$ with $\llbracket R_1 \rrbracket=\llbracket R \rrbracket\cap \{\blacktriangle X\in \Gamma\}$ and $\llbracket R_2 \rrbracket=\llbracket R \rrbracket\cap \{\blacktriangle X\notin\Gamma\}$, if $\FF'_R$ is independent of $\mathbbm{1}_\Gamma(\Delta_RX)$ under $\PP'$, and if there is a law $\MM$ on $(\mathbb{D},\mathcal{D})$ satisfying  $\PP'(\Delta_RX\in \Gamma)=\MM(\Gamma)$, $(\Delta_{R_1}X)_\star \PP^1=\MM(\cdot\vert \Gamma)$ and $(\Delta_{R_2}X)_\star \PP^2=\MM(\cdot\vert \mathbb{D}\backslash \Gamma)$, then $\FF_R'$ is independent of $\Delta_RX$ under $\PP'$ and $(\Delta_RX)_\star \PP'=\MM$.
\end{enumerate}
\end{enumerate}
\end{lemma}
\begin{remark}\label{remark:restriction}
\leavevmode
\begin{enumerate}
\item \label{remark:restriction:i}
Recall $(\blacktriangle X)_t=\Delta_tX=(X_{t+u}-X_t)_{u\in [0,\infty)}$ for $t\in [0,\infty)$, and note that $\blacktriangle X\in \mathcal{G}\otimes \mathcal{B}_{[0,\infty)}/\mathcal{D}$. 
\item\label{remark:restriction:ii} In terms of $\mathcal{O}$ and $\sigma_{\Omega\times [0,\infty)}(\blacktriangle X)$, the property of $R$ being an independence time for $X$ can be rephrased as follows. Let $\mu^{\PP,R}$ be the unique probability measure $\mu$ on $\mathcal{G}\otimes \mathcal{B}_{[0,\infty)}$ satisfying $\mu(Z)=\PP'( Z_R)$ for $Z\in \mathcal{G}\otimes \mathcal{B}_{[0,\infty)}/\mathcal{B}_{[0,\infty]}$. Then $\Delta_R X$ is independent of $\mathcal{F}_R'$ under $\PP'$ iff $\mathcal{O}$ is independent of $\sigma_{\Omega\times [0,\infty)}(\blacktriangle X)$ under $\mu^{\PP,R}$. 
\end{enumerate}
\end{remark}
\begin{proof}[Proof of Lemma~\ref{lemma:restriction}]
Take $Z\in \Opt/\mathcal{B}_{[0,\infty]}$ and $H\in \mathcal{D}/\mathcal{B}_{[0,\infty]}$.  \ref{restriction:I}\ref{restriction:a}. The fact that $R$ is an independence time for $X$, together with the optionality of the process $Z\mathbbm{1}_{A}$, yields $\PP(R<\infty)\PP(Z_SH(\Delta_SX);S<\infty)=\PP(Z_S;S<\infty)\PP(H(\Delta_RX);R<\infty)$. 
Setting $Z\equiv 1$ and plugging back in, everything follows. \ref{restriction:I}\ref{restriction:b}. Now $H\mathbbm{1}_\Gamma\in \mathcal{D}/\mathcal{B}_{[0,\infty]}$ and one proceeds in a similar fashion. \ref{restriction:II}\ref{restriction:i}. We have $\PP((Z\mathbbm{1}_A)_RH(\Delta_RX);R<\infty)=\PP((\mathbbm{1}_AZ)_R;R<\infty)\MM(H)$ and $\PP((Z\mathbbm{1}_{A^c})_RH(\Delta_RX);R<\infty)=\PP((\mathbbm{1}_{A^c}Z)_R;R<\infty)\MM(H)$.
Summing the two, the desired conclusion follows upon taking $Z\equiv 1$ and plugging it back in. The proof of \ref{restriction:II}\ref{restriction:ii} is similar. 
\end{proof}
Define now the family $\mathfrak{F}$ of collections $\mathcal{T}$ of $\FF$-stopping times, identified up to a.s. equality, as follows: $\mathcal{T}\in \mathfrak{F}$ if and only if $\mathcal{T}$ is a collection of (equivalence classes mod $\PP$) of $\FF$-stopping times such that the following two conditions are met:
\begin{itemize}\label{collection:T}
\item if $T_1\in \mathcal{T}$, $T_2\in \mathcal{T}$, and $\PP(T_1\ne T_2)\ne 1$, then $\PP(T_1=T_2<\infty)=0$; and 
\item if $T\in \mathcal{T}$, then $\PP(R=T)>0$. 
\end{itemize}
Then $\mathfrak{F}$ is non-empty, is partially ordered by inclusion, and every linearly ordered subset of $\mathfrak{F}$ admits an upper bound. By Zorn's lemma there is a maximal element $\mathcal{T}\in \mathfrak{F}$. Since $\PP(R<\infty)\leq 1<\infty$, $\mathcal{T}$ is denumerable. By the maximality of $\mathcal{T}$, it follows that $A:=\cup_{T\in \mathcal{T}}\llbracket T\rrbracket$ is an optional set for which $\llbracket R\rrbracket A^c\llbracket S\rrbracket=\emptyset$ up to evanescence for each $\FF$-stopping time $S$. Let the random times $R_1$ and $R_2$ be defined by $\llbracket R_1\rrbracket=\llbracket R\rrbracket A$ and $\llbracket R_2\rrbracket=\llbracket R\rrbracket A^c$. Then the graph of $R_1$ is included up to evanescence in the union of some denumerable family of $\FF$-stopping times ($R_1$ is a `thin random time'), whilst $R_2$ satisfies $\llbracket R_2\rrbracket \llbracket S\rrbracket=\emptyset$ up to evanescence  for each $\FF$-stopping time $S$ ($R_2$ is a `strict random time'). Moreover, when both $\PP(R_1<\infty)$ and $\PP(R_2<\infty)$ are positive, then according to Lemma~\ref{lemma:restriction}, $R$ is an independence time for $X$ if and only if $R_1$ and $R_2$ are both independence times for $X$ with $\Delta_{R_1}X$ having the same law on $\{R_1<\infty\}$ as does $\Delta_{R_2}X$ on $\{R_2<\infty\}$. 

We (may thus) deal with thin and strict random times separately. The former are described completely in:

\begin{proposition}\label{proposition:necessity}
\leavevmode
\begin{enumerate}[(I)]
\item\label{proposition:necessity:i} Suppose $\FF'_R$ is independent of $\Delta_R X$ under $\PP'$. Then, for any $\FF$-stopping time $S$, satisfying $\PP(R=S<\infty)>0$ and $\{R=S<\infty\}\in \overline{\FF_S'\lor \sigma_{\{S<\infty\}}(\Delta_SX)}$, there are an a.s. unique $F_S\in \FF_S'$ and an $\LL$-a.s. unique $G_S\in \mathcal{D}$, such that a.s. $\{R=S<\infty\}=F_S\{\Delta_SX\in G_S\}$, in which case furthermore $(\Delta_RX)_\star\PP'=\PP(X\in \cdot\vert X\in G_S)=\LL(\cdot\vert G_S)$. In particular, if $S_1$ and $S_2$ are two such $\FF$-stopping times, then $\LL$-a.s. $G_{S_1}=G_{S_2}$. 
\item\label{proposition:necessity:ii} If a.s. $\{R<\infty\}=\cup_{i\in \mathbb{N}}\{R=S_i<\infty\}$ for a sequence $(S_i)_{i\in \mathbb{N}}$ of $\FF$-stopping times, and if there are a $G\in \mathcal{D}$, and for each $i\in \mathbb{N}$ an $F_i\in \mathcal{F}_{S_i}'$, satisfying $\{R=S_i<\infty\}=F_i \{\Delta_{S_i}X\in G\}$ a.s., then  $\FF'_{R+}$ is independent of $\Delta_R X$ under $\PP'$ and $(\Delta_RX)_\star\PP'=\PP(X\in \cdot\vert X\in G)=\LL(\cdot\vert G)$. 
\end{enumerate}
In particular, if $R\in \overline{\sigma(X)}/\mathcal{B}_{[0,\infty]}$ and $R$ is a thin random time, then $\FF'_R$ is independent of $\Delta_R X$ under $\PP'$ iff there exist $O\in \mathcal{O}$ and $G\in \mathcal{D}$ such that $\llbracket R\rrbracket=O\{\blacktriangle X\in G\}$ up to evanescence, and when so then even $\FF'_{R+}$ is independent of $\Delta_R X$ under $\PP'$ with $(\Delta_RX)_\star\PP'=\LL(\cdot\vert G)$.
\end{proposition}
The proof of Proposition~\ref{proposition:necessity} will follow on p.~\pageref{proof:thin} after we have given some remarks and examples.
\begin{remark}
\leavevmode
\begin{enumerate}
\item For sure $\{R=S<\infty\}\in \overline{\FF_S'\lor \sigma_{\{S<\infty\}}(\Delta_SX)}$ for all $\FF$-stopping times $S$, if $R\in \overline{\sigma(X)}/\mathcal{B}_{[0,\infty]}$. 
 \item $X$ is also a L\'evy process with respect to the usual augmentation of $\FF$. 
\item It is conjectured that there is no independence time for $X$ belonging to $\overline{\sigma(X)}/\mathcal{B}_{[0,\infty]}$, that is finite and equal to a stopping time of $X$ with a positive probability, yet whose graph fails to be included, up to evanescence, in the union of the graphs of a denumerable family of stopping times of the right-continuous augmentation of $\FF$. 
\end{enumerate}
\end{remark}
%
\begin{example}\label{example:thin-times}
Let $X$ be a linear Brownian motion with strictly negative drift and continuous sample paths, vanishing at zero. Denote by $\overline{X}$ the running supremum of $X$ and by $L$ the last time that $X$ is at its running supremum (set $L=\infty$ on the negligible event  that there is no such last time). Let $R$ be equal to $0$ on the event that $X$ never reaches the level $1$;  or else let it be equal to the first hitting time of $-1$ by the process $\Delta_LX$ on the event $\{L<\infty\}$; $R=\infty$ otherwise. For $q\in \mathbb{Q}_{>0}$ let $T_q$ be the first entrance time into the set $\{1\}$ of the process $\overline{X}-X$ \emph{strictly after} time $q$, and let $S_q:=(T_q)_{\{X_s\ne \overline{X}_s\text{ for }s\in (q,T_q)\}}$. Then $\{R<\infty\}=\cup_{q\in \mathbb{Q}_{>0}}\{R=S_q<\infty\}\cup \{R=0\}$, $\{R=0\}=\{X\text{ never reaches }1\}$ and for $q\in \mathbb{Q}_{>0}$, letting $L_q$ be the last time $X=\overline{X}$ on $[0,q]$, $\{R=S_q<\infty\}=\{S_q<\infty\}\{1\leq \overline{X}_{q}>X_s>\overline{X}_{q}-1\text{ for }s \in (L_q,S_q)\}\{\Delta_{S_q}X\text{ never reaches }1\}$.  By Proposition~\ref{proposition:necessity}\ref{proposition:necessity:ii} $R$ is an independence time for $X$ (even with $\FF'_R$ replaced by $\FF'_{R+}$). 

\end{example}

\begin{example}\label{example:thin-times:W-H}
Let $X$ be real-valued ($d=1$), let $\overline{X}$ be the running supremum of $X$, and assume $0$ is irregular for itself for the drawdown (a.k.a. reflected in the supremum, regret) process $\overline{X}-X$ (equivalently, $0$ is irregular for $[0,\infty)$ for the process $X$). Recalling Example~\ref{example:WH} it will again be interesting to see how in this case the independence property of the Wiener-Hopf factorization for L\'evy processes (see \cite[Lemma~VI.6(i)]{bertoin}) falls out of the above. Let indeed $N$ be a homogeneous Poisson process of rate $\lambda\in [0,\infty)$, vanishing at zero, adapted to $\FF$, independent of $X$, and such that the pair $(X,N)$ is a L\'evy process relative to $\FF$. (The case $\lambda=0$ corresponds to the zero process $N=0$ a.s..) Note that $e:=\inf\{t\in [0,\infty):N_t>0\}$ is an $\FF$-stopping time independent of $X$ having the exponential distribution with mean $\lambda^{-1}$ ($e=\infty$ a.s. when $\lambda=0$). Let finally $R=\sup\{t\in [0,e):X_t=\overline{X}_t\}$. We insists of course that $\PP(R<\infty)>0$ (equivalently, in this case, $\PP(R<\infty)=1$; the latter being automatic when $\lambda>0$, whilst it obtains iff $X$ drifts to $-\infty$ when $\lambda=0$). It is clear that under the assumption on the drawdown process the successive visits of $\overline{X}-X$ to $0$ form a sequence of $\FF$-stopping times whose graphs exhaust the graph of $R$ up to evanescence; $R$ is a thin random time. Moreover, up to evanescence, $$\llbracket R\rrbracket=\{\overline{X}=X\}\llbracket 0,e\rrparenthesis\{\text{the first component of } \blacktriangle (X,N)\text{ does not return to zero up to strictly before}$$ $$\text{the first time that the second component of }\blacktriangle (X,N)\text{ increases}\}.$$ It follows that Proposition~\ref{proposition:necessity} applies and, under $\PP'$, yields the independence of $\FF'_{R+}$ (in particular of the stopped process $X^R$) from $\Delta_R(X,N)$ (in particular from $\Delta_R X$). Furthermore, the distribution of $\Delta_RX$ on $\{R<\infty\}$ is that of $X$ conditioned not to return to its running supremum until strictly before an independent exponential random time of rate $\lambda$.
\end{example}
\begin{proof}[Proof of Proposition~\ref{proposition:necessity}]\label{proof:thin}
\ref{proposition:necessity:i}. Uniqueness of $G_S$ and $F_S$ is a consequence of Lemma~\ref{lemma} and of the strong Markov property of $X$. Set $\PP^:=(\mathcal{G}\vert_{\{S<\infty\}}\ni A\mapsto \PP(A\vert S<\infty))$. We verify that $\PP^:$-a.s. $\PP^:(R=S)\mathbbm{1}(R=S)=\PP^:(R=S\vert \FF_S')\PP^:(R=S\vert \sigma_{\{S<\infty\}}(\Delta_SX))$. Since $\{R=S<\infty\}\in \overline{\FF_S'\lor \sigma_{\{S<\infty\}}(\Delta_SX)}$ and since $\FF'_S/\mathcal{B}_{\{0,1\}}=\{Z_S:Z\in \mathcal{O}/\mathcal{B}_{\{0,1\}}\}$, by monotone class and the strong Markov property of $X$, this will follow, if we have shown that for all $Z\in \mathcal{O}/\mathcal{B}_{[0,\infty]}$ and  $H\in \mathcal{D}/\mathcal{B}_{[0,\infty]}$, $\PP^:(R=S)\PP^:(Z_SH(\Delta_SX);R=S)=\PP^:(Z_S;R=S)\PP^:(H(\Delta_SX);R=S)$.  But this follows from Lemma~\ref{lemma:restriction}\ref{restriction:I}\ref{restriction:a} applied to the optional set $\llbracket S\rrbracket$. By Lemma~\ref{lemma}, this establishes the existence of $G_S$ and $F_S$. Next we compute, for any $H\in \mathcal{D}/\mathcal{B}_{[0,\infty]}$,\footnotesize
$$\PP^:(R=S)\PP(H(X);X\in G_S)=\PP^: (F_S)\PP^:(\Delta_SX\in G_S)\PP^:(H(\Delta_SX);\Delta_SX\in G_S)=\PP^: (H(\Delta_SX);F_S \{\Delta_SX\in G_S\})\PP(X\in G_S)$$ $$=\PP^:(H(\Delta_SX);R=S)\PP(X\in G_S)=\PP^:(H(\Delta_RX);R=S)\PP(X\in G_S)=\PP(X\in G_S)\frac{\PP^:(R=S)\PP(H(\Delta_RX);R<\infty)}{\PP(R<\infty)},$$ \normalsize 
where the final equality is a consequence of the independence of $\Delta_R X$ and $\FF'_R$ under $\PP'$ (as applied to the optional process $\mathbbm{1}_{\llbracket S\rrbracket}$).
Similarly as in the proof of Theorem~\ref{theorem}, the rest of the claims of \ref{proposition:necessity:i} follow.

\ref{proposition:necessity:ii}. Let $S$ be an $\FF$-stopping time. Then a.s. $\{R=S<\infty\}=\cup_{i\in \mathbb{N}}\{R=S<\infty\} \{S=S_i\}=
\cup_{i\in \mathbb{N}}F_i\{S=S_i\}\{\Delta_{S}X\in G\}$ where $\cup_{i\in \mathbb{N}}F_i\{S=S_i\}\in \mathcal{F}_S'$. Hence \cite[Theorem~3.31]{semimtgs} we may assume without loss of generality that the graphs of the $S_i$, $i\in \mathbb{N}$, are pairwise disjoint. Then we may compute for $Z\in \mathrm{Prog}/\mathcal{B}_{[0,\infty]}$ and $H\in \mathcal{D}/\mathcal{B}_{[0,\infty]}$, using the strong Markov property of $X$:\footnotesize $$\PP(Z_RH(\Delta_RX);R<\infty)=\sum_{i\in \mathbb{N}}\PP(Z_{S_i}H(\Delta_{S_i}X);F_i\{\Delta_{S_i}X\in G\})
=\sum_{i\in \mathbb{N}}\PP(Z_{S_i};F_{i}\{S_i<\infty\})\PP(H(X);X\in G)$$ \normalsize and we conclude by taking $Z\equiv 1$, $H\equiv 1$ and $Z\equiv 1\equiv  H$ in turn. 

For the final observation of the proposition, we note as follows. By definition, $R$ being a thin random time means that there exists a denumerable family of $\FF$-stopping times $(S_i)_{i\in \mathcal{I}}$, the graphs of the members of which may be assumed pairwise disjoint, and such that $\llbracket R\rrbracket\subset \cup_{i\in \mathbb{N}}\llbracket S_i\rrbracket$ up to evanescence and $\PP(R=S_i)>0$ for all $i\in \mathcal{I}$. Then if $R$ is an independence time for $X$, we may use \ref{proposition:necessity:i}, setting $O=\cup_{i\in \mathcal{I}}\llbracket S_i\rrbracket(F_{S_i}\times [0,\infty))$ and $G$ equal to any of the $G_{S_i}$s. This establishes the necessity of the condition. Sufficiency follows from \ref{proposition:necessity:ii}, upon taking $F_i\in \FF'_{S_i}$, so that $\mathbbm{1}_{F_i}=(\mathbbm{1}_O)_{S_i}\mathbbm{1}(S_i<\infty)$ for $i\in \mathcal{I}$. 
\end{proof}

Strict independence times appear to be more subtle. We give below a sufficient condition for a strict random time to be an independence time for $X$ (Proposition~\ref{proposition:sufficient:two}). Some preliminary notions and results are needed to this end.

\begin{definition}
A $T\in \mathcal{G}/\mathcal{B}_{[0,\infty]}$ is an incremental terminal time (ITT) if, for all $t\in [0,\infty)$ with $\PP(T>t)>0$, conditionally on $\{T>t\}$, $T-t$ is independent of $\FF_t\vert_{\{T>t\}}$ and has the same law as $T$. 
\end{definition}
\begin{remark}\label{remark:ITT}
\leavevmode
\begin{enumerate}
\item\label{remark:ITT:i} If $T$ is an ITT, then: by the characterization of the solutions to Cauchy's functional equation, $\PP(T=0)=1$, \emph{or} $\PP(T=\infty)=1$,  \emph{or} $T$ has an exponential distribution; furthermore, for all $\FF$-stopping times $S$, satisfying $\PP(T>S)>0$, via the usual extension from deterministic to stopping times, conditionally on $\{T>S\}$, $T-S$ is independent of $\FF_S\vert_{\{T>S\}}$ and has the same law as $T$. 
\item\label{remark:ITT:ii}  ITTs are the analogue of terminal times (e.g. \cite[Definition~4.1]{sharpe-getoor}) from the Markov context: If $T$ is an $\FF$-stopping time for which there is a $K\in \mathcal{D}/\mathcal{B}_{[0,\infty]}$ with $T=t+K(\Delta_tX)$ a.s. on $\{T>t\}$ for each $t\in [0,\infty)$, then by the simple Markov property of $X$, $T$ is an ITT. 
\item\label{remark:ITT:iii} If $T$ is an ITT and $e\in \mathcal{G}/\mathcal{B}_{[0,\infty]}$ is an exponentially distributed random time independent of $\FF_\infty\lor \sigma(T)$ then $e\land T$ is an ITT.
\end{enumerate}
\end{remark}

\begin{examples}
ITTs that fall under Remark~\ref{remark:ITT}\eqref{remark:ITT:ii} include the deterministic times $\infty$ and $0$ and, assuming $\FF$ is right-continuous --- this is not a major assumption, since $X$ is a L\'evy process also relative to the right-continuous augmentation of $\FF$ --- for an open $O\subset \mathbb{R}^d$, the time of the first entrance into (or hitting of) the set $O$ by the jump process $\Delta X$ of $X$ (\cite{sokol} proves that such hitting times are stopping times when $d=1$ -- a perusal of the proof given, convinces one that the dimension $d=1$ has nothing special to it in this regard; when $\FF$ is not right-continuous, then still these hitting times are ITTs, though they may fail to be stopping times). Indeed, by the D\'ebut theorem \cite[Theorem~4.2]{semimtgs},  when $\GG$ is universally complete, any first entrance time of the jump process of $X$ into a set from $\mathcal{B}_{\mathbb{R}^d}$, is an ITT (and an $\FF$-stopping time if further $(\Omega,\GG,\FF,\PP)$ satisfies ``the usual hypotheses'' \cite[Theorem~4.30]{semimtgs}).
\end{examples}

\begin{definition}
Let $T$ be an ITT. A process $A\in \mathcal{G}\otimes \mathcal{B}_{[0,\infty)}/\mathcal{B}_{[0,\infty)}$, nondecreasing and right-continuous, is an incremental functional (IF) up to $T$ if (i) $A_0=0$ a.s., (ii) $A=A^T$ a.s. 
and (iii)
for all $t\in [0,\infty)$, if $\PP(T>t)>0$, then, conditionally on $\{T>t\}$, $(\Delta_t X,\Delta_tA)$ is independent of $\FF_t\vert_{\{T>t\}}$ and has the same law as $(X,A)$. 
\end{definition}
\begin{remark}\label{remark:IF}
\leavevmode
\begin{enumerate}
\item If $A$ is an IF up to $T$ and $H\in \mathcal{D}/\mathcal{B}_{[0,\infty]}$ is such that for all $t\in [0,\infty)$, $\tilde{A}_t:=\int_{[0,t]}H(\Delta_s X) dA_s$ is finite-valued, then $\tilde{A}:=(\tilde{A}_t)_{t\in [0,\infty)}$ is an IF up to $T$. 
\item\label{remark:IF:ii} 
If $A\in \mathcal{G}\otimes\mathcal{B}_{[0,\infty)}/\mathcal{B}_{[0,\infty)}$ is nondecreasing, right-continuous, vanishing a.s. at zero, if $T$ is an $\FF$-stopping time that is an ITT rendering $A^T=A$ a.s., and if there is a $J\in \mathcal{D}/\mathcal{D}$ such that $\Delta_tA=J(\Delta_tX)$ a.s. on $\{T>t\}$ for all $t\in [0,\infty)$, then by the simple Markov property of $X$, $A$ is an IF up to $T$. IFs are the analogue of raw additive functionals (e.g. \cite[Definition~4.2]{sharpe-getoor}) from the Markov context.
\item\label{remark:IF:iii}  If $A$ is an IF up to $T$, then in fact for all $\FF$-stopping times $S$, satisfying $\PP(T>S)>0$, conditionally on $\{T>S\}$, $(\Delta_S X,\Delta_SA)$ is independent of $\FF_S\vert_{\{T>S\}}$ and has the same law as $(X,A)$. This is proved in a manner that is entirely analogous to the proof of the strong Markov property of $X$ from the simple Markov property of $X$ (see e.g. \cite[proof of Theorem~40.10]{sato}). 
\item\label{remark:IF:iv}   If $A$ is an IF up to $T$ and $e\in \mathcal{G}/\mathcal{B}_{[0,\infty]}$ is an exponentially distributed random time independent of $\FF_\infty\lor\sigma(A,T)$, then the stopped process $A^e$ is an IF up to $T\land e$ (cf. Remark~\ref{remark:ITT}\eqref{remark:ITT:iii}). 
\end{enumerate}
\end{remark}

\begin{lemma}\label{lemma:IF}
Let $T$ be an ITT, $A$ an IF up to $T$, $H\in \mathcal{D}/\mathcal{B}_{[0,\infty]}$ and $M\in \Opt/\mathcal{B}_{[0,\infty]}$. Set $\lambda:=(\PP T)^{-1}\in [0,\infty]$. Assume that (i) for all $N\in (0,\infty)$, $\PP\int_{(0,\epsilon]}H(\Delta_uX)\land NdA_u<\infty$ for some $\epsilon\in (0,\infty)$ and that (ii) $\PP((\Delta A)_S;S<\infty)=0$ for all $\FF$-stopping times $S$ ($\Delta A$ being the jump process of $A$). Then $$\PP\int_{[0,\infty)} M_uH(\Delta_uX)dA_u=\PP\int_0^1H(\Delta_uX)dA_u\PP\int_0^T M_udu\frac{\lambda}{1-e^{-\lambda}}
$$ with the quotient understood in the limiting sense when $\lambda\in \{0,\infty\}$. (It is part of the statement that the inclusion or the exclusion of the upper delimiter $1$ in the integral $\PP\int_0^1\ldots dA_u$ is immaterial.)
\end{lemma}
%
%
\begin{proof}
The case $\PP(T=0)=1$ is trivial; assume $\PP(T=0)<1$. By Remark~\ref{remark:IF}\eqref{remark:IF:iv} and monotone convergence, we may assume without loss of generality (possibly by exploiting an extension of the underlying filtered probability space on which there lives a mean one exponentially distributed random time $e$ independent of $\FF_\infty\lor\sigma(A,T)$) that $\PP(T<\infty)=1$ (one minimizes $T$ by $ne$, stops $A$ at $T\land (ne)$, and sends $n\to\infty$). By Remark~\ref{remark:ITT}\eqref{remark:ITT:i} there is then a $\lambda\in (0,\infty)$ with $\PP(T>t)=e^{-\lambda t}$ for all $t\in [0,\infty)$. By monotone convergence we may  assume that $M$ are $H$ bounded and hence (by assumption (i)) that $\PP\int_{(0,\epsilon]}H(\Delta_uX)dA_u<\infty$ for some $\epsilon\in (0,\infty)$. 
%

Define $f_H:=([0,\infty)\ni t\mapsto  \PP\int_{(0,t]}H(\Delta_uX)dA_u)$. For $t\in [0,\infty)$, $s\in [t,\infty)$, we have \footnotesize 
\begin{equation}\label{eq:fundamental}
\PP \left(\int_{(t,s]}H(\Delta_uX)dA_u\right)=\PP \left(\int_{(0,s-t]}H(\Delta_u\Delta_tX)d(\Delta_tA)_u;t<T\right)=\PP(t<T)\PP\int_{(0,s-t]}H(\Delta_uX)dA_u.\end{equation}\normalsize
We find that $f_H$ is finite-valued, nondecreasing, and satisfies the functional equation $f_H(s)-f_H(t)=e^{-\lambda t}f_H(s-t)$. Set $s=t+1$, to see that the limit $f_H(\infty):=\lim_{\infty}f_H$ is finite, then send $s\to \infty$, to obtain $f_H(\infty)-f_H(t)=e^{-\lambda t}f_H(\infty)$. It follows that $f_H(\infty)=\frac{f_H(1)\lambda}{1-e^{-\lambda}}\PP T$. 

It now follows from Remarks~\ref{remark:ITT}\eqref{remark:ITT:i} and~\ref{remark:IF}\eqref{remark:IF:iii}, and from assumption (ii), that when $M=\mathbbm{1}_{\llbracket S,\infty\rrparenthesis}$ for an $\FF$-stopping time $S$, then $$\PP\int_{[0,\infty)} M_uH(\Delta_uX)dA_u=\PP\left(\int_{(S,\infty)}H(\Delta_uX)dA_u;S<T\right)=\PP\left(\int_{(0,\infty)} H(\Delta_u\Delta_SX)d(\Delta_SA)_u;S<T\right)$$ $$=\PP(S<T)\PP T\frac{f_H(1)\lambda}{1-e^{-\lambda}}=\PP(T-S;S<T)\frac{f_H(1)\lambda}{1-e^{-\lambda}}=\frac{\lambda}{1-e^{-\lambda}}\PP\int_0^1H(\Delta_uX)dA_u\PP\int_0^T M_udu.$$
The class of processes $M$ of the form considered is closed under multiplication and generates the optional $\sigma$-field \cite[Theorem~3.17]{semimtgs}. By monotone class we conclude. 
\end{proof}

\begin{proposition}\label{proposition:sufficient:two}
Suppose there exist $T$, an ITT, and $A$, an IF up to $T$, satisfying $\PP(A_\epsilon)<\infty$ for some (then all) $\epsilon\in(0,\infty)$ and $\PP((\Delta A)_S;S<\infty)=0$ for all $\FF$-stopping times $S$ ($\Delta A$ being the jump process of $A$). Suppose furthermore that there exists an $O\in\mathcal{O}/\mathcal{B}_{[0,\infty)}$, such that there is the following equality of random measures on $\mathcal{B}_{[0,\infty)}$: 
\begin{equation}\label{eq:LP-independence:two}
\delta_R\mathbbm{1}(R<\infty)=O\cdot dA \text{ a.s.},
\end{equation}
where $\delta_R$ is the Dirac mass at $R$ and $(O\cdot dA)(C):=\int_C O_sdA_s$ for $C\in \mathcal{B}_{[0,\infty)}$. Then $\FF_R'$ is independent of $\Delta_RX$ under $\PP'$, and the law of $\Delta_RX$ under $\PP'$ is given as follows: for $H\in \mathcal{D}/\mathcal{B}_{[0,\infty]}$, $\PP'(H(\Delta_RX))=\PP\int_0^1H(\Delta_uX)dA_u/\PP A_1$. 
\end{proposition}
The proof of Proposition~\ref{proposition:sufficient:two} will follow on p.~\pageref{proof:thin-random-times} after we have given  some remarks and examples. 
\begin{remark}\label{remark:sufficient:two}
\leavevmode
\begin{enumerate}
\item The assumptions of Proposition~\ref{proposition:sufficient:two} imply $\llbracket R\rrbracket \llbracket S\rrbracket =\emptyset$ up to evanescence for any $\FF$-stopping time $S$. In particular, they are are mutually exclusive with those of Proposition~\ref{proposition:necessity}\ref{proposition:necessity:ii}.

\item\label{cts-sufficient-rmks:ii} In the setting of Proposition~\ref{proposition:sufficient:two}, the `canonical' situation to have in mind, is the one in which $T=\infty$ a.s. and there is a $\Gamma\in \mathcal{D}$, with the jump process $\Delta A$ of $A$ satisfying $\{\Delta A>0\}=\{\blacktriangle X\in \Gamma\}$ (see the examples following). Note that in such case the assumptions of Proposition~\ref{proposition:sufficient:two} imply $\llbracket R\rrbracket=\{O>0\}\{\blacktriangle X\in \Gamma\}$ up to evanescence.

\item The presence of the ITT $T$ in the statement of Proposition~\ref{proposition:sufficient:two} allows some non-trivial added generality. For instance, independent exponential killing: if for the random time $R$, an IF $A$ up to some $T$  and an $O$ have been found satisfying the assumptions of Proposition~\ref{proposition:sufficient:two}, and if $e\in \mathcal{G}/\mathcal{B}_{[0,\infty]}$ is an exponentially distributed time independent of $\FF_\infty\lor \sigma(A,T)$, then $A^e$ and $O$ also satisfy the assumptions of Proposition~\ref{proposition:sufficient:two}, but for $R_{\{R<e\}}$ replacing $R$ (see Remark~\ref{remark:IF}\eqref{remark:IF:iv}). 

\item \eqref{eq:LP-independence:two} is a continuous-time analogue of \eqref{eq:RW-independence}: indeed the latter can be written as $\delta_R\mathbbm{1}(R<\infty)=Z\cdot \mu$ with $\mu=\sum_{n\in \mathbb{N}_0}\delta_n\mathbbm{1}_{\Gamma}(\Delta_n X)$ and $Z=\sum_{n\in \mathbb{N}_0}\mathbbm{1}_{\llbracket n\rrbracket}\mathbbm{1}_{F_n}$. 

\item Another way of writing \eqref{eq:RW-independence} is as $\llbracket R\rrbracket = O\{\blacktriangle X\in \Gamma\}$ up to evanescence, with $O=\cup_{n\in \mathbb{N}_0}F_n\times \{n\}$. This, item \eqref{cts-sufficient-rmks:ii}, and the final statement of Proposition~\ref{proposition:necessity}, suggest the following as being the most natural analogue of  \eqref{eq:RW-independence}  in continuous time: $\llbracket R\rrbracket=O \{\blacktriangle X\in \Gamma\}$ up to evanescence for an $O\in \mathcal{O}$ and a $\Gamma\in \mathcal{D}$. Establishing the precise relationship (conjectured equivalence, when $R\in \overline{\sigmađ(X)}/\mathcal{B}_{[0,\infty]}$) between the property of $R$ being an independence time for $X$ and the latter condition, escapes the author in general (Proposition~\ref{proposition:necessity} establishes this equivalence for thin random times).  Note that a `set-theoretic' condition of the sort ``$\llbracket R\rrbracket=O \{\blacktriangle X\in \Gamma\}$'' should typically lend itself well to a check in application. For instance, when $X$ is a L\'evy process for which $0$ is transient and regular for itself, and $R$ is the last time $X$ is equal to $0$, then one may simply write (since $X$ is a.s. continuous at the left (and right) end-points of its zero free intervals \cite[Eq.~(2.9)]{millar-germ}) $\llbracket R\rrbracket=\{X=0\}\{\blacktriangle X\vert_{\Omega\times (0,\infty)}\ne 0\}$ (up to evanescence).
\end{enumerate}
\end{remark}
\begin{example}
Let $X$ be the difference between a strictly positive unit drift and a non-zero compound Poisson subordinator whose paths are all piecewise constant (to avoid measurability issues). Let $R$ be the unique time at which $X$ is one temporal unit away from jumping for the second time and $X$ is above zero, if there is such a time, $R=\infty$ otherwise. Note that $\PP(R<\infty)>0$. Now let $A$ be the right-continuous nondecreasing process vanishing a.s. at zero which increases by one at precisely those times $t\in [0,\infty)$ when $\Delta_tX$ performs the following: drifts for one spatial unit, and then jumps. By Remark~\ref{remark:IF}\eqref{remark:IF:ii}, $A$ is an IF up to $\infty$. Clearly $\PP A_1<\infty$ and by the strong Markov property of $X$ it follows that $\PP((\Delta A)_S;S<\infty)=0$ for all $\FF$-stopping times $S$. Let next $N$ be the right-continuous adapted process that counts the number of jumps of $X$ by, inclusive of, a given time; finally let $O:=\mathbbm{1}(N=1,X\geq 0)$ record the times when $X$ has had one jump to-date and is not below zero. Then $O\in \Opt/2^{\{0,1\}}$ and \eqref{eq:LP-independence:two} obtains. So $R$ is an independence time for $X$. 
\end{example}

\begin{example}\label{example:at-sup}
This is a complement to Example~\ref{example:thin-times:W-H}. 
Assume $d=1$ and $X$ is sample path continuous (so a linear Brownian motion with drift) and suppose $\FF$ is the completed natural filtration of $X$. Let $e$ be an independent exponentially with parameter $\lambda\in [0,\infty)$ distributed random time ($e=\infty$ a.s. when $\lambda=0$) and let $R=\sup\{t\in [0,e):X_t=\overline{X}_t\}$ be the last time on $[0,e)$ that $X$ is at its running supremum $\overline{X}$.  We assume of course the drift of $X$ is strictly negative when $\lambda=0$ so that $\PP(R<\infty)>0$ (and indeed $=1$). For $\epsilon\in (0,\infty)$, let $A^\epsilon$ be the right-continuous nondecreasing process vanishing a.s. at zero that increases by $1$ at precisely those $t\in [0,e)$ for which $\Delta_tX$ hits the level $\epsilon$ before hitting zero, and then does not return to $\epsilon$ strictly before time $e-t$. Since between any two increases of $A^\epsilon$, $X$ must ascend by $\epsilon$ and then decrease by at least $\epsilon$, it follows that the set of such times is locally finite in $[0,\infty)$ and that moreover $\PP(A_1^\epsilon)<\infty$. By the strong Markov property of $X$, $\PP((\Delta A)_S^\epsilon;S<\infty)=0$ for all $\FF$-stopping times $S$. By Remark~\ref{remark:ITT}\eqref{remark:ITT:iii} $e$ is an ITT. By the simple Markov property of $X$ and the memoryless property of $e$,  $A^\epsilon$ is an IF up to $e$. Set $O^\epsilon:=\mathbbm{1}(X+\epsilon\geq \overline{X})$. Clearly $O^\epsilon\in \Opt/2^{\{0,1\}}$. Let $R^\epsilon$ be equal to the last time on $[0,e)$ that we are, in spatial terms, $\epsilon$ away $\overline{X}_e$ ($R^\epsilon=\infty$, when there is no such time). By Proposition~\ref{proposition:sufficient:two} we obtain that $R^\epsilon$ is an independence time for $X$. In particular, for all continuous bounded $Z\in \mathcal{O}/\mathcal{B}_\mathbb{R}$, all $n\in \mathbb{N}$, all real $0<t_1<\cdots <t_n$, and all bounded continuous $h:\mathbb{R}^n\to \mathbb{R}$: \footnotesize $$\PP(R^\epsilon<\infty)\PP(Z_{R^\epsilon} h((\Delta_{R^\epsilon}X)_{t_1},\ldots, (\Delta_{R^\epsilon}X)_{t_n});R^\epsilon<\infty)=\PP(Z_{R^\epsilon};R^\epsilon<\infty)\PP(h((\Delta_{R^\epsilon}X)_{t_1},\ldots,(\Delta_{R^\epsilon}X)_{t_n});R^\epsilon<\infty).$$\normalsize
As a.s. $R_\epsilon\to R$ as $\epsilon\downarrow 0$, we may pass to the limit. Finally, since in a Brownian filtration optionality is synonymous with predictability, a monotone class argument allows to conclude that $R$ is an independence time for $X$ (cf. \cite[Lemma~VI.6(ii)]{bertoin}). 
\end{example}
\begin{proof}[Proof of Proposition~\ref{proposition:sufficient:two}]\label{proof:thin-random-times}
Set $\lambda:=(\PP T)^{-1}\in [0,\infty)$. Let $M\in \mathcal{O}/\mathcal{B}_{[0,\infty]}$ and $H\in \mathcal{D}/\mathcal{B}_{[0,\infty]}$. Then by Lemma~\ref{lemma:IF}, $$\PP(M_RH(\Delta_RX);R<\infty)=\PP\left(\int M_uH(\Delta_uY)\delta_R(du);R<\infty\right)=$$ 
$$\PP\left(\int_{[0,\infty)} M_uH(\Delta_uX)O_udA_u\right)=\PP\int_0^1H(\Delta_uX)dA_u\PP\int_0^T M_uO_udu\frac{\lambda}{1-e^{-\lambda}}.$$One concludes as in the proof of  Proposition~\ref{proposition:necessity}\ref{proposition:necessity:ii}. 
\end{proof}

\begin{remark}
The preceding proposition was inspired by, and should be compared with \cite[Theorem~5.7]{sharpe-getoor}. There is an extra possibility in the latter, which in our case would correspond to  $\delta_R\mathbbm{1}(R<\infty)$ having a mixture of the two forms, \eqref{eq:LP-independence:two} and the one from Proposition~\ref{proposition:necessity}\ref{proposition:necessity:ii}, by separating them according to a subset $M\subset \Omega\times [0,\infty)$ that is both optional and also of the form $\{\blacktriangle X\in  \Gamma\}$ for a $\Gamma\in \mathcal{D}$. But if $M$ is such a set, then, as $S$ ranges over the class of finite $\FF$-stopping times, $\mathbbm{1}_M(S)=\mathbbm{1}_\Gamma(\Delta_SX)$ is at the same time $\FF_S$ measurable and, by the strong Markov property, independent of $\FF_S$ and of constant expectation. It is to say that the process $\mathbbm{1}_M$ is a.s. equal to zero at every finite stopping time $S$ or is a.s. equal to one at every finite stopping time $S$. Hence any such $M$ equals $\emptyset$ or $\Omega\times [0,\infty)$ up to evanescence \cite[Corollary~4.11]{semimtgs}, and the separation would be trivial (see, however, Lemma~\ref{lemma:restriction}\ref{restriction:II}). Indeed, the separation in \cite[Theorem~5.7]{sharpe-getoor} happens according to an optional homogeneous (in the sense of Markov processes, e.g. \cite[p. 315]{sharpe-getoor}) set and, by contrast, there are many non-trivial sets of such a form (e.g. $\{X\in \Gamma\}$ for $\Gamma\in \mathcal{B}_{\mathbb{R}^d}$).
\end{remark}

\begin{remark}\label{remark:markov-connect}
In special cases, the `conditional independence' results from the Markov theory (see Subsection~\ref{subsection:lit-overview}) may also be applied  successfully  in the present context  -- modulo said results' provisions. Either directly, when $X_R$ is trivial, or indirectly when there is some other strong Markov process $Z$ for which $Z_R$ is trivial and for which $R$ is a conditional independence time if and only if it is an independence time for $X$.  For instance, $X$ reflected in its running supremum is known to be strong Markov and a sample-path continuous process $˛Z$ is, for $a$ from its state space, equal to $a$ on its last visit to $a$ (viz. Example~\ref{example:at-sup}).
\end{remark}


We conclude by characterizing regenerative/Markov times as stopping times. 

\begin{lemma}\label{lemma:generators}
Assume $\mathcal{G}=\sigma(X)$ (respectively, up to negligible sets). Then  $\mathcal{G}\otimes \mathcal{B}_{[0,\infty)}=\mathcal{O}\lor \sigma_{\Omega\times [0,\infty)}(\blacktriangle X)$ (respectively, up to evanescent sets). 
\end{lemma}
\begin{proof}
Since every optional process is measurable and by Remark~\ref{remark:restriction}\eqref{remark:restriction:i}, $\mathcal{G}\otimes \mathcal{B}_{[0,\infty)}\supset \mathcal{O}\lor \sigma_{\Omega\times [0,\infty)}(\blacktriangle X)$. For the reverse inclusion, by monotone class, it suffices to show that for $n\in \mathbb{N}_0$, real numbers $0= r_0<\cdots <r_n$, $\{\alpha_0,\cdots,\alpha_n\}\subset \mathbb{R}^d$ and $r\in [0,\infty)$, the process $Z:=\prod_{k=0}^n e^{i \alpha_k\cdot X_{r_k}}\mathbbm{1}_{\llbracket r,\infty\rrparenthesis}$ belongs to $\mathcal{O}\lor \sigma_{\Omega\times [0,\infty)}(\blacktriangle X)/\mathcal{B}_{\mathcal{\mathbb{C}}}$. Set $r_{n+1}:=\infty$ and let $j\in \{0,\ldots,n\}$ be the unique index for which $r\in [r_j,r_{j+1})$. Conclude by identifying $Z_t=\mathbbm{1}_{[r,\infty)}(t)\left(e^{i(\alpha_{j+1}+\cdots+ \alpha_n)\cdot X_r}\prod_{k=0}^je^{i \alpha_k\cdot X_{r_k}}\right)\left(\prod_{k=j+1}^ne^{i \alpha_k\cdot ((\blacktriangle X)_{r\land t}(r_k-r))}\right)$ for $t\in [0,\infty)$ and noting that deterministic stopping is $\mathcal{D}/\mathcal{D}$-measurable.
\end{proof}

\begin{theorem}\label{theorem:characterization:STs:cts}
For $x\in \mathbb{R}^d$, denote by $\PP^x$ be the law of $x+X$ under $\PP$ (so $\LL=\PP^0$), $\mathcal{P}:=(\PP^x)_{x\in \mathbb{R}^d}$. Of the following statements, \ref{cts:ii} implies \ref{cts:i}, \ref{cts:o} and \ref{cts:i} are equivalent. If in addition $R\in \overline{\sigma(X)}/\mathcal{B}_{[0,\infty]}$, then \ref{cts:i} implies \ref{cts:ii}.

\begin{enumerate}[(i)]
\item \label{cts:o} Under $\PP'$, conditionally on $X_R$, $\theta_R X$ is independent of $\FF_R'$ and has kernel law $\mathcal{P}$. 
\item\label{cts:i} $\Delta_RX$ is independent of $\FF_R'$ under $\PP'$ and the law of $\Delta_RX$ under $\PP'$  is $\LL$.
\item\label{cts:ii} $R$ is a stopping time relative to the completion of $\FF$. 
\end{enumerate}
\end{theorem}

\begin{proof}
The equivalence of  \ref{cts:o} and \ref{cts:i}  follows by routine manipulation of conditional independence, using the fact that $X_R\in \FF'_R/\mathcal{B}_{\mathbb{R}^d}$. Condition \ref{cts:ii} is known to be sufficient for \ref{cts:i}. Necessity. Since the completed natural filtration of $X$ is right-continuous (every L\'evy process is a Feller process), we may assume without loss of generality that the filtered probability space $(\Omega,\GG,\PP,\FF)$ satisfies the ``usual hypotheses'' of the general theory of stochastic processes with $\FF$ the completed natural filtration of $X$ and $\GG=\overline{\sigma(X)}$. Consider the measure $\mu_R$ on $\mathcal{G}\otimes \mathcal{B}_{[0,\infty)}$ associated to the raw increasing process $\mathbbm{1}_{\llbracket R,\infty\rrparenthesis}$: $\mu_RM=\PP(M_R;R<\infty)$ for $M\in \mathcal{G}\otimes \mathcal{B}_{[0,\infty)}/\mathcal{B}_{[0,\infty]}$. It suffices \cite[Theorem~5.13]{semimtgs} to establish $\mu_R$ is an optional measure, i.e. that for $M\in \mathcal{G}\otimes \mathcal{B}_{[0,\infty)}/\mathcal{B}_{[0,\infty]}$, $\mu_RM=\mu_R({}^oM)$ with ${}^oM$ the optional projection of $M$. By monotone class and Lemma~\ref{lemma:generators}, it suffices to check it when $M=ZH(\blacktriangle X)$ with $Z\in \mathcal{O}/\mathcal{B}_{[0,\infty]}$ and $H\in \mathcal{D}/\mathcal{B}_{[0,\infty]}$. In such case, by the strong Markov property of $X$, for all $\FF$-stopping times $S$, a.s. on $\{S<\infty\}$, one has $\PP(M_S\vert \FF_S)=\PP(Z_SH(\Delta_S X)\vert \FF_S)=Z_S\LL(H)$, i.e. ${}^oM=Z\LL(H)$. Now, since $R$ is an independence time for $X$, indeed $\mu_R M=\PP(Z_R H(\Delta_RX);R<\infty)=\PP(Z_R;R<\infty)\LL(H)=\mu_R({}^oM)$.  
\end{proof}
\bibliographystyle{amsplain}
\bibliography{Biblio_strong_markov}

\end{document}